\documentclass[10pt]{amsart}

\usepackage{amsmath,amsthm,amsfonts,latexsym,amssymb,mathrsfs,setspace,enumerate,verbatim}

\newtheorem{thm}{Theorem}

\newtheorem{lem}[thm]{Lemma}

\newtheorem{cor}[thm]{Corollary}
\numberwithin{thm}{section}
\numberwithin{equation}{section}

\theoremstyle{definition}

\newtheorem*{ex}{Example}

\newcommand{\rat}{\mathbb Q}

\newcommand{\com}{\mathbb C}
\newcommand{\alg}{\overline\rat}
\newcommand{\algt}{\alg^{\times}}

\newcommand{\intg}{\mathbb Z}
\newcommand{\nat}{\mathbb N}

\newcommand{\B}{\mathcal B}
\newcommand{\tor}{\mathrm{Tor}}

\newcommand{\rad}{\mathrm{Rad}}
\newcommand{\spa}{\mathrm{Span}}

\title{The finiteness of computing the ultrametric Mahler measure}
\author[C.L. Samuels]{Charles L. Samuels}
\address{The University of British Columbia, Department of Mathematics, 1984 Mathematics Road, Vancouver, BC V6T 1Z2, Canada}
\email{csamuels@math.ubc.ca}

\begin{document}

\begin{abstract}
	Recent work of Fili and the author examines an ultrametric version of the Mahler measure, denoted $M_\infty(\alpha)$ for an algebraic number $\alpha$.  
	We show that the computation of $M_\infty(\alpha)$ can be reduced to a certain search through a finite set.
  Although it is a open problem to record the points of this set in general, 
	we provide some examples where it is reasonable to compute and our result can be used to determine $M_\infty(\alpha)$.
\end{abstract}

\maketitle

\section{Introduction and Notation}

Let $K$ be a number field and $v$ a place of $K$ dividing the place $p$ of $\rat$.  Let $K_v$ and $\rat_p$
denote the respective completions.  We write $\|\cdot\|_v$ for the unique absolute value on $K_v$ 
extending the $p$-adic absolute value on $\rat_p$ and define 
\begin{equation*}
  |\alpha|_v=\|\alpha \|_v^{[K_v:\rat_p]/[K:\rat]}
\end{equation*}
for all $\alpha\in K$.  Define the {\it Weil height} of $\alpha\in K$ by
\begin{equation*} \label{WeilHeightDef}
  H(\alpha) = \prod_v\max\{1,|\alpha|_v\}
\end{equation*}
where the product is taken over all places $v$ of $K$.  Given this normalization of our absolute values,
the above definition does not depend on $K$, and therefore, $H$ is a well-defined function on $\alg$.

For the remainder of the paper, we shall assume that $\alpha$ is a non-zero algebraic number.
Clearly $H(\alpha)\geq 1$, and by Kronecker's Theorem, we have equality precisely when $\alpha$ a root of unity.  
It is obvious that if $\zeta$ is a root of unity then
\begin{equation} \label{HeightRoot}
	H(\alpha) = H(\zeta\alpha),
\end{equation}
and further, if $n$ is an integer then it is well-known that
\begin{equation*} 
	H(\alpha^n) = H(\alpha)^{|n|}.
\end{equation*}
Also, if $\alpha,\beta\in\algt$ then $H(\alpha\beta) \leq H(\alpha)H(\beta)$ so that $H$ 
satisfies the multiplicative triangle inequality.

We further define the {\it Mahler measure} of $\alpha$ by 
\begin{equation} \label{MahlerDef}
	M(\alpha) = H(\alpha)^{[\rat(\alpha):\rat]}.
\end{equation}
Since $H$ is invariant under Galois conjugation over $\rat$, we obtain immediately
\begin{equation*} 
	M(\alpha) = \prod_{n=1}^N H(\alpha_n),
\end{equation*}
where $\alpha_1,\ldots,\alpha_N$ are the conjugates of $\alpha$ over $\rat$.  Further, it is well-known that
\begin{equation} \label{SecondMahlerDefinition}
	M(\alpha) = |A|\cdot\prod_{n=1}^N\max\{1,|\alpha_n|\},
\end{equation}
where $|\cdot |$ denotes the usual absolute value on $\com$ and $A$ is the leading coefficient of the minimal polynomial of $\alpha$ over $\intg$.  
While the right hand side of \eqref{SecondMahlerDefinition} appears initially to depend upon a particular 
embedding of $\alg$ into $\com$, any change of embedding simply permutes the images of the points $\{\alpha_n\}$ so that \eqref{SecondMahlerDefinition} 
remains unchanged.  In view of \eqref{SecondMahlerDefinition}, it is reasonable to define the {\it Mahler measure of a polynomial}
\begin{equation*}
	f(z) = A\cdot\prod_{n=1}^N(z-\alpha_n) \in \com[x]
\end{equation*}
by
\begin{equation} \label{PolynomialMahler}
	M(f) = |A|\cdot\prod_{n=1}^N\max\{1,|\alpha_n|\}.
\end{equation}
We note that if $f$ is the minimal polynomial of $\alpha$ over $\intg$, then $M(f) = M(\alpha)$, so that \eqref{PolynomialMahler} is compatible with
\eqref{MahlerDef}.

It follows, again from Kronecker's Theorem, that $M(\alpha) = 1$ if and only if $\alpha$ is a root of unity.  As part of an algorithm for
computing large primes, D.H. Lehmer \cite{Lehmer} asked whether there exists a constant $c>1$ such that $M(\alpha) \geq c$ in all other cases.
The smallest known Mahler measure greater than $1$, already found by Lehmer, occurs at a root of 
\begin{equation*}
  \ell(x) = x^{10}+x^9-x^7-x^6-x^5-x^4-x^3+x+1
\end{equation*}
which has Mahler measure $1.17\ldots$.  
Although an affirmative answer to Lehmer's problem has been given in many special cases, the general case remains open.  The best known universal lower 
bound on $M(\alpha)$ is due to Dobrowolski \cite{Dobrowolski}, who proved that 
\begin{equation*} 
  \log M(\alpha) \gg \left(\frac{\log\log \deg \alpha}{\log \deg \alpha}\right)^3
\end{equation*}
whenever $\alpha$ is not a root of unity.

In \cite{DubSmyth}, Dubickas and Smyth \cite{DubSmyth} defined the {\it metric Mahler measure} of $\alpha$ by
\begin{equation} \label{MetricDef}
	M_1(\alpha) = \inf\left\{\prod_{n=1}^NM(\alpha_n):N\in\nat,\ \alpha_n\in\algt,\ \alpha = \prod_{n=1}^N\alpha_n\right\}.
\end{equation}
Here, the infimum is taken over all ways to represent $\alpha$ as a product of elements in $\algt$.
It is easily verified that 
\begin{equation*}
	M_1(\alpha\beta) \leq M_1(\alpha)M_1(\beta)
\end{equation*} 
for all $\alpha,\beta\in\algt$.   

Further, we write
\begin{equation*}
	V = \algt/\tor(\algt)
\end{equation*}
and note that $V$ is a vector space over $\rat$.  The scalar multiplication in $V$ is given by the maps $\alpha\mapsto\alpha^r$, which is well-defined 
on $V$ for any $r$ in $\rat$.  The operation of the vector space is multiplication in $V$ and the identity element is $\tor(\algt)$.

It is a simple exercise to verify that $M_1$ is well-defined on $V$.  This implies that the map
$(\alpha,\beta)\mapsto\log M_1(\alpha\beta^{-1})$ defines a metric on $V$ which induces the discrete topology if and only if
there is an affirmative answer to Lehmer's problem.

Motivated by the work of Dubickas and Smyth, Fili and the author \cite{FiliSamuels} defined a non-Archimedean version of $M_1$ by
replacing the product in \eqref{MetricDef} by a maximum.  Define the {\it ultrametric Mahler measure} by
\begin{equation*}
	M_\infty(\alpha) = \inf\left\{\max_{1\leq n\leq N}M(\alpha_n):N\in\nat,\ \alpha_n\in\algt,\ \alpha = \prod_{n=1}^N\alpha_n\right\}.
\end{equation*}
It easily verified that $M_\infty$ satisfies the strong triangle inequality 
\begin{equation*}
	M_\infty(\alpha\beta)\leq\max\{M_\infty(\alpha),M_\infty(\beta)\}
\end{equation*}
for all non-zero algebraic numbers $\alpha$ and $\beta$.  It is further shown in \cite{FiliSamuels} that $M_\infty$ is well-defined on $V$.

The goal of this paper is to reduce the computation of $M_\infty(\alpha)$ to a certain search through a finite set.  In order to do this, we are
required to work in $V$, a space on which $M$ is not well-defined.  Hence, we define the {\it modified Mahler measure} by
\begin{equation*} 
	\bar M(\alpha) = \inf\left\{M(\zeta\alpha):\zeta\in\tor(\algt)\right\}.
\end{equation*}
It is immediately clear that $\bar M$ is well-defined on $V$.
Further, if $\pi:\algt\to V$ denotes the natural group homomorphism and $\bar\alpha\in V$, then we obtain
\begin{equation*} 
	\bar M(\bar\alpha) = \inf\left\{M(\beta):\beta\in\pi^{-1}(\bar\alpha)\right\}.
\end{equation*}
We further conclude using \eqref{HeightRoot} and \eqref{MahlerDef} that
\begin{equation} \label{MBarPower}
	\bar M(\alpha) = H(\alpha)^{\inf\{\deg(\zeta\alpha):\zeta\in\tor(\algt)\}},
\end{equation}
which implies, in particular, that there always exists a root of unity $\zeta$ such that $M(\zeta\alpha) = \bar M(\alpha)$.
Some additional basic properties of $\bar M$ will be examined in section \ref{ModifiedMeasure}.

Next, let $K_\alpha$ denote the Galois closure of $\rat(\alpha)$ over $\rat$ and define the set
\begin{equation*}
	\B(\alpha) = \{\gamma\in K_\alpha: M(\gamma) \leq M(\alpha)\}.
\end{equation*}
We note that, by a theorem of Northcott \cite{Northcott}, $\B(\alpha)$ is finite.  Then let
\begin{equation*} 
	\bar \B(\alpha) = \pi(\B(\alpha)).
\end{equation*}
We may now express the ultrametric Mahler measure in terms of a certain point in $\bar \B(\alpha)$.

\begin{thm} \label{Computation}
	Let $\alpha$ be a non-zero algebraic number and let $\bar\B(\alpha) = \{\bar b_1,\ldots,\bar b_N\}$.  Assume that
	\begin{equation} \label{Ordering}
		\bar M(\bar b_1)\leq \bar M(\bar b_2)\leq \cdots\leq \bar M(\bar b_N).
	\end{equation}
	If $J$ is the smallest index such that 
	\begin{equation*}
		\pi(\alpha) \in \spa\{\bar b_1,\cdots,\bar b_J\}
	\end{equation*}
	then $M_\infty(\alpha) = \bar M(\bar b_J)$.
\end{thm}

In some cases, Theorem \ref{Computation} enables us to compute $M_\infty(\alpha)$.  The related calculations will be simpler using the following restatement of
Theorem \ref{Computation}.

\begin{cor} \label{SimpleComputation}
	Let $\alpha$ be a non-zero algebraic number and let $\B = \{b_1,\ldots,b_N\} \subseteq \B(\alpha)$ be such that $\pi(\B) = \bar \B(\alpha)$.  Assume that
	\begin{equation} \label{SimpleOrdering}
		\bar M(b_1)\leq \bar M(b_2)\leq \cdots\leq \bar M(b_N).
	\end{equation}
	If $J$ is the smallest index such that there exists a positive integer $s$ with 
	\begin{equation*}
		\alpha^s \in \langle b_1,\ldots,b_J \rangle
	\end{equation*}
	then $M_\infty(\alpha) = \bar M(b_J)$.
\end{cor}
	
We will present a more involved example in section \ref{Applications}, but we note one simple application here.  In \cite{FiliSamuels}, Fili and the author showed, 
using another method, that $M_\infty(4) = 2$.  We are able to recover this observation using Theorem \ref{Computation}.  We have that
\begin{equation*}
	\B(4) = \left\{\pm 1, \pm 2, \pm 3, \pm 4, \pm \frac{1}{2}, \pm \frac{3}{2}, \pm \frac{1}{3}, \pm \frac{2}{3}, \pm \frac{4}{3},
	    \pm \frac{1}{4}, \pm \frac{3}{4}\right\}.
\end{equation*}
Now let
\begin{equation*}
	\B = \left\{ 1, 2, 3, 4, \frac{1}{2}, \frac{3}{2}, \frac{1}{3}, \frac{2}{3}, \frac{4}{3}, \frac{1}{4}, \frac{3}{4}\right\}.
\end{equation*}
and note that $\pi(\B) = \bar \B(4)$.
For $r\in \rat$, it follows from \eqref{MBarPower} that $\bar M(r) = M(r)$.  Now we may rewrite $\B$ with its elements written
in increasing order of modified Mahler measure.  We obtain
\begin{equation*}
	\B = \left\{ 1, 2,\frac{1}{2}, 3, \frac{1}{3}, \frac{2}{3}, \frac{3}{2}, 4, \frac{1}{4}, \frac{3}{4}, \frac{4}{3}\right\}.
\end{equation*}
We observe that $4 = 1^0 2^2$.  However, there cannot exist an integer $s > 0$ such that $4^s \in \langle 1 \rangle$ since $4$ is not a root of unity.  
By Corollary \ref{SimpleComputation}, we conclude that $M_\infty(4) = 2$.

This argument can be used to recover a more general statement from \cite{FiliSamuels}.

\begin{cor} \label{RationalAlg}
	If $\alpha$ is rational, then $M_\infty(\alpha)$ equals the largest prime dividing the numerator or denominator of $\alpha$.
\end{cor}

The remainder of this paper is structured as follows.  Section \ref{Applications} contains an additional example of Theorem \ref{Computation}
when $\deg\alpha = 2$.  Although we know of no general formula analogous to Corollary \ref{RationalAlg}, the quadratic situation is simple enough that 
some explicit computations can be made.  In section \ref{ModifiedMeasure}, we examine the properties of $\bar M$ that we will need in 
order to prove our main results.  Finally, we use sections \ref{MainProof} and \ref{ApplicationsProofs} to establish Theorem \ref{Computation} as well as
prove some results related to the applications in section \ref{Applications}.

\section{Further Applications} \label{Applications}

Although Theorem \ref{Computation} is of theoretical interest, we would like to apply it to compute values of $M_\infty(\alpha)$.  Initially, this seems quite 
reasonable since Theorem \ref{Computation} reduces the computation of $M_\infty(\alpha)$ to a search over a finite set.  However, there remain three obstacles to 
performing such a computation.\\

\begin{enumerate}[I.]
\item\label{BObstacle} We must determine an appropriate set $\B\subseteq \B(\alpha)$ for use in Corollary \ref{SimpleComputation}.  While there are explicit upper bounds on the 
cardinality of $\B(\alpha)$ in terms of $\deg\alpha$ (see, for example \cite{LoherMasser}), we are unaware of an algorithm for recording these points.
As we will discuss in section \ref{ApplicationsProofs}, there is a highly inefficient method for recording all \emph{polynomials} of Mahler measure at most $M(\alpha)$, a collection
whose roots clearly belong to $\B(\alpha)$.  However, even this is insufficient since these polynomials may not be solvable.

\item\label{ModifiedObstacle} The modified Mahler measures of points in $\B$ need to be computed with sufficient accuracy to write these points in increasing order of modified Mahler measures.
In view of \eqref{MBarPower}, we know that $\bar M(\alpha) = M(\zeta\alpha)$ where $\zeta$ is a root of unity that makes $\deg(\zeta\alpha)$ as small as possible.
Unfortunately, we know of no general method to locate a suitable element $\zeta$.

\item\label{LocatingObstacle} We must locate the point $b_J$, as required by Corollary \ref{SimpleComputation}.  Again, we do not know of a general method for doing so.
\end{enumerate}

If $\alpha$ is a quadratic number with $M(\alpha) \leq 100$, then we are able to use PARI \cite{PARI} to construct a set of polynomials whose roots form a suitable set 
$\B$ for use in Corollary \ref{SimpleComputation}.  If the list is not too long, then it is a simple exercise to record the roots of these polynomials, resolving \eqref{BObstacle}.
Further, we are able to give formulae for $\bar M(\gamma)$ when $\deg\gamma \leq 2$, which resolves \eqref{ModifiedObstacle}.
We provide an example where are able to resolve \eqref{LocatingObstacle} as well and compute the value $M_\infty(\alpha)$.

For simplicity, we will now write
\begin{equation*}
	\B_d(\alpha) = \B(\alpha) \cap \{\alpha\in\alg: \deg\alpha = d\}
\end{equation*}
and note that
\begin{equation} \label{BUnion}
	\B(\alpha) = \bigcup_{d\mid [K_\alpha:\rat]} \B_d(\alpha)
\end{equation}
where the right hand side is a disjoint union.  From elementary facts about the Mahler measure, we notice that
\begin{equation} \label{RationalB}
	\B_1(\alpha) = \left\{\pm \frac{m}{n}: m,n \in \nat,\ (m,n) = 1,\ \max\{|m|,|n|\} \leq M(\alpha) \right\}
\end{equation}
so it is a simple exercise to obtain the points of $\B_1(\alpha)$.

Since we are now interested in the quadratic case, we will assume for the remainder of this section that $\deg\alpha\leq 2$.
To resolve \eqref{BObstacle} in this situation, we must compute a subset $\B''\subset \B_2(\alpha)$ such that $\pi(\B'') = \pi(\B_2(\alpha))$.
Before we write a PARI program to do this, we must write a program to estimate the Mahler measure of a quadratic polynomial $ax^2 + bx + c$.

\begin{verbatim}
Mahler(a,b,c) =
{
  abs(a)*max(1,abs((-b+sqrt(b^2-4*a*c))/(2*a)))
             *max(1,abs((-b-sqrt(b^2-4*a*c))/(2*a)));
}
\end{verbatim}

While {\tt Mahler(a,b,c)} is only an estimate of $M(ax^2+bx+c)$, it is reasonable to assume some level of accuracy.  For our purposes, we will assume that
\begin{equation*} \label{EstimateAmount0}
	|M(ax^2+bx+c) - {\tt Mahler(a,b,c)}| < 10^{-10}.
\end{equation*}
Indeed, it is quite reasonable to assume that PARI will compute accurately up to at least $10$ decimal places.

Next suppose that $\alpha$ is a root of the irreducible polynomial $Ax^2+Bx+C$ and that $k$ is the unique square-free integer such that $K_\alpha = \rat(\sqrt k)$.
We claim that the following simple program can be used to find a suitable set $\B$ for use in Corollary \eqref{SimpleComputation}.

\begin{verbatim}
B2List(A,B,C,k) = 
{
  local(M);

  M = Mahler(A,B,C) + 10^(-10);

  for(a=1,floor(M),
    for(b=0,floor(2*M),
      for(c=-floor(M),floor(M),
        if(!issquare(b^2-4*a*c) && gcd(a,gcd(b,c)) ==1 && 
           issquare((b^2-4*a*c)/k) && Mahler(a,b,c) < M + 10^(-10), 
           printp("("a","b","c") -- "Mahler(a,b,c)));
      );
    );
  );
}
\end{verbatim}

The above program searches all triples $(a,b,c)$ satisfying
\begin{equation*}
	1 \leq a \leq M(\alpha),\ 0 \leq b \leq 2M(\alpha),\ \mathrm{and}\ -M(\alpha) \leq c \leq M(\alpha).
\end{equation*}
For each such point, it checks, up to some computing error, if
\begin{enumerate}[(i)]
	\item $b^2-4ac$ is not a perfect square and $\gcd(a,b,c) = 1$ (i.e., $ax^2+bx+c$ is irreducible),
	\item $(b^2 - 4ac)/k$ is a perfect square (i.e., the roots of $ax^2+bx+c$ belong to $\rat(\sqrt k)$),
	\item $M(ax^2+bx+c) \leq M(\alpha)$.
\end{enumerate}
If the above three conditions are satisfied, the the program prints $(a,b,c)$ with $M(ax^2+bx+c)$ alongside.  Otherwise, it prints nothing.

Our next theorem shows that the output list of {\tt B2List(A,B,C,k)} can, indeed, be used to construct an appropriate set $\B$ for use in 
Corollary \ref{SimpleComputation}, provided that $M(\alpha)$ is not too large.

\begin{thm} \label{RestrictedListComputing}
	Suppose $\alpha$ is a quadratic number with minimal polynomial $Ax^2+Bx+C$ and $M(\alpha) \leq 100$.  Let $k$ be the unique square-free integer
	such that $K_\alpha = \rat(\sqrt k)$.  Assume that
	\begin{equation} \label{EstimateAmount}
		|M(ax^2+bx+c) - {\tt Mahler(a,b,c)}| < 10^{-10}
	\end{equation}
	holds for all integers $a, b$ and $c$ with $a > 0$.  Suppose that $\B''$ is the set of all roots of the polynomials $ax^2+bx+c$, where
	$(a,b,c)$ appears in the output of {\tt B2List(A,B,C,k)}.  Further write
	\begin{equation*}
		\B' = \left\{\frac{m}{n}: m,n \in \nat,\ (m,n) = 1,\ \max\{|m|,|n|\} \leq M(\alpha) \right\}.
	\end{equation*}
	If $\B = \B'\cup\B''$ then $\B\subset \B(\alpha)$ and $\pi(\B) = \bar \B(\alpha)$.
\end{thm}

We now turn our attention to resolving \eqref{ModifiedObstacle}.  We already have a PARI function {\tt Mahler(a,b,c)} that approximates the Mahler measure of
a quadratic polynomial.  So we must reduce the computation of $\bar M$ to a computation of $M$.  The following theorem shows how to do this in the quadratic case.

\begin{thm} \label{MBarCalcs}
	Suppose that $\gamma\in\alg$ with $\deg\gamma \leq 2$.
	\begin{enumerate}[(i)]
	\item\label{GeneralCase} If $\rat(\gamma) \ne \rat(i)$ and $\rat(\gamma) \ne \rat(i\sqrt 3)$ then $\bar M(\gamma) = M(\gamma)$.
	\item\label{ICase} If $\gamma = a+bi$ for rational numbers $a$ and $b$ then
	\begin{equation*}
		\bar M(\gamma) = \left\{
		\begin{array}{ll}
			\displaystyle M(b) & \mathrm{if}\ a=0 \\
			\displaystyle M(\gamma) & \mathrm{if}\ a \ne 0.
		\end{array}
		\right.
	\end{equation*}
	\item\label{3rdCase} If $\gamma = a +b\sqrt{-3}$ for rational numbers $a$ and $b$ then
	\begin{equation*}
		\bar M(\gamma) = \left\{
		\begin{array}{ll}
			\displaystyle M(2a) & \mathrm{if}\ a\in\{b,-b\} \\
			\displaystyle M(\gamma) & \mathrm{if}\ a \not\in \{b,-b\}.
		\end{array}
		\right.
	\end{equation*}
	\end{enumerate}	
\end{thm}

As we have noted, we know of no method that resovles \eqref{LocatingObstacle} for general quadratic numbers.  However, in the following example, $b_J$ can be found, and hence, 
$M_\infty(\alpha)$ can be computed.

\begin{ex}
	We take
	\begin{equation*}
		\alpha = \frac{5 + \sqrt{21}}{2}
	\end{equation*}
	so that $k = 21$.  We note that $\alpha$ has minimal polynomial $x^2 - 5x + 1$ and Mahler measure
	\begin{equation*}
		M(\alpha) = \frac{5 + \sqrt{21}}{2} = 4.791287847477920003294023597
	\end{equation*}
	as computed by {\tt Mahler(1,-5,1)}.  Executing {\tt B2List(1,-5,1,21)} yields the output
	\vskip4mm
	\begin{verbatim}
	(1, 3, -3) -- 3.791287847477920003294023597
	(1, 5, 1) -- 4.791287847477920003294023597
	(3, 3, -1) -- 3.791287847477920003294023597
	\end{verbatim}
	\vskip4mm
	Hence, the set $\B''$ from Theorem \ref{RestrictedListComputing} is given by
	\begin{equation*}
		\B'' = \left\{ -\frac{3\pm \sqrt{21}}{2}, -\frac{5 \pm \sqrt{21}}{2}, -\frac{3\pm \sqrt{21}}{6} \right\}
	\end{equation*}
	and we also have
	\begin{equation*}
		\B' = \left\{ 1, 2, 3, 4, \frac{1}{2}, \frac{3}{2}, \frac{1}{3}, \frac{2}{3}, \frac{4}{3}, \frac{1}{4}, \frac{3}{4}\right\}.
	\end{equation*}
	According to Theorem \ref{RestrictedListComputing}, we may set
	\begin{align} \label{FirstBApp}
	\B & = \left\{1, 2,\frac{1}{2}, 3, \frac{1}{3}, \frac{2}{3}, \frac{3}{2}, -\frac{3 + \sqrt{21}}{2}, -\frac{3 - \sqrt{21}}{2}, -\frac{3 + \sqrt{21}}{6},\right. \nonumber \\
						&\qquad\qquad \left.-\frac{3 - \sqrt{21}}{6}, 4, \frac{1}{4}, \frac{3}{4}, \frac{4}{3},-\frac{5 + \sqrt{21}}{2}, -\frac{5 - \sqrt{21}}{2} \right\}.
	\end{align}
	and we have that $\B \subset \B(\alpha)$ such that $\pi(\B) = \bar \B(\alpha)$.  By Theorem \ref{MBarCalcs}, we have that $M(\gamma) = \bar M(\gamma)$ 
	for all $\gamma\in K_\alpha$, and therefore, \eqref{FirstBApp} is already recorded in increasing order of modified Mahler measures.  A short computation reveals that
	\begin{equation*}
		\frac{5 + \sqrt{21}}{2} = \frac{1}{3}\cdot\left( -\frac{3 + \sqrt{21}}{2}\right)^2
	\end{equation*}
	so that
	\begin{equation*}
		\alpha \in \left\langle 1, 2,\frac{1}{2}, 3, \frac{1}{3}, \frac{2}{3}, \frac{3}{2}, -\frac{3 + \sqrt{21}}{2} \right\rangle.
	\end{equation*}
	However, if there exists a positive integer $s$ such that
	\begin{equation*}
		\alpha^s \in \left\langle 1, 2,\frac{1}{2}, 3, \frac{1}{3}, \frac{2}{3}, \frac{3}{2}\right\rangle
	\end{equation*}
	then $\alpha^s$ is rational, which is a contradiction using the binomial theorem.  It follows from Corollary \ref{SimpleComputation} that
	\begin{equation*}
		M_\infty\left(	\frac{5 + \sqrt{21}}{2}\right) = M\left(-\frac{3 + \sqrt{21}}{2}\right) = \frac{3 + \sqrt{21}}{2} = 3.791287847...
	\end{equation*}
\end{ex}

\section{The modified Mahler measure} \label{ModifiedMeasure}

In our proof of Theorem \ref{Computation}, we will often be required to consider the modified Mahler measure rather than the classical 
Mahler measure.  In this section, we establish some basic properties that relate these two functions.
Our first lemma establishes a basic inequality regarding powers of algebraic numbers in the function $\bar M$.

\begin{lem} \label{IntegerPower}
	If $\bar\alpha\in V$ and $L$ is a positive integer then $\bar M(\bar\alpha) \leq \bar M(\bar\alpha^L)$.
\end{lem}
\begin{proof}
	If $\gamma$ is any algebraic number, then it is well-known that $\deg\gamma \leq L\deg\gamma^L$.  Hence, it follows that
	\begin{equation*}
		M(\gamma) = H(\gamma)^{\deg\gamma} \leq H(\gamma)^{L\deg \gamma^L} = H(\gamma^L)^{\deg\gamma^L} = M(\gamma^L).
	\end{equation*}
	We now have that
	\begin{align*}
		\bar M(\bar\alpha) & = \inf\{M(\zeta\alpha):\zeta\in\tor(\algt)\} \\
			& \leq \inf\{M((\zeta\alpha)^L):\zeta\in\tor(\algt)\} \\
			& = \inf\{M(\zeta\alpha^L):\zeta\in\tor(\algt) \\
			& = \bar M(\alpha^L).
	\end{align*}
	However, we know that $\bar M(\alpha^L) = \bar M(\pi(\alpha^L)) = \bar M(\bar\alpha^L)$ completing the proof.
\end{proof}

It will be very natural in the proof of Theorem \ref{Computation} to consider the strong metric version of $\bar M$.  In other words, If $\bar\alpha\in V$, we define
\begin{equation*}
	\bar M_\infty(\bar\alpha) = \inf\left\{\max_{1\leq n\leq N}\bar M(\bar\alpha_n):N\in\nat,\ \bar\alpha_n\in V,\ \bar\alpha = \prod_{n=1}^N\bar\alpha_n\right\}.
\end{equation*}
As we have stated, our goal for this article is to reduce the computation of $M_\infty(\alpha)$, not $\bar M_\infty(\pi(\alpha))$, to a finite set.  
However, we cannot expect that $M(\alpha) = \bar M(\pi(\alpha))$ in general.  For example, $M(2i) = 4$ while $\bar M(2i) = 2$.
Therefore, it is not immediate that $M_\infty(\alpha) = \bar M_\infty(\pi(\alpha))$ for any algebraic number $\alpha$.  However, the following lemma
shows that these two functions are indeed equal.

\begin{lem} \label{Conversion}
	If $\alpha$ is a non-zero algebraic number then $M_\infty(\alpha) = \bar M_\infty(\pi(\alpha))$.
\end{lem}
\begin{proof}
	We see immediately that $\bar M(\pi(\gamma)) \leq M(\gamma)$ for all $\gamma\in\alg$, so it follows that $\bar M_\infty(\pi(\alpha)) \leq M_\infty(\alpha)$.
	
	To prove the opposite inequality, let $\bar\alpha_1,\ldots,\bar\alpha_N \in V$ be such that $\pi(\alpha) = \bar\alpha_1\cdots\bar\alpha_N$.
	Since the infimum in the definition of $\bar M$ is attained, for each $n$ there exist points $\alpha_n\in\pi^{-1}(\bar\alpha_n)$ such that $M(\alpha_n) = \bar M(\bar\alpha_n)$.
	Therefore, we have that
	\begin{equation*}
		\pi(\alpha) = \pi(\alpha_1\cdots\alpha_N) = \pi(\alpha_1)\cdots\pi(\alpha_N) 
	\end{equation*}
	which implies the existence of a root of unity $\zeta$ such that
	\begin{equation*}
		\alpha = \zeta\alpha_1\cdots\alpha_N.
	\end{equation*}  
	Hence, we obtain
	\begin{equation*}
		M_\infty(\alpha) \leq \max\{M(\zeta),M(\alpha_1),\ldots,M(\alpha_N)\} = \max\{\bar M(\bar\alpha_1),\ldots,\bar M(\bar\alpha_N)\}.
	\end{equation*}
	The result follows by taking the infimum of both sides over all factorizations $\pi(\alpha) = \bar\alpha_1\cdots\bar\alpha_N$.
\end{proof}

We now write
\begin{equation*}
	\rad(K_\alpha) = \left\{\gamma\in\algt: \gamma^n\in K_\alpha\ \mathrm{for\ some}\ n\in\nat\right\}.
\end{equation*}
The author showed (see \cite{Samuels}, Theorem 2.1) that, for any representation $\alpha=\alpha_1\cdots\alpha_N$, there exists another
representation $\alpha = \zeta\beta_1\cdots\beta_N$ with $\zeta$ a root of unity, $M(\beta_n)\leq M(\alpha_n)$
and $\beta_n\in\rad(K_\alpha)$ for all $n$.  In particular, as we attempt to compute the value of $M_\infty(\alpha)$ in general,
we need only consider representations of $\alpha$ in $\rad(K_\alpha)$.  

In view of Lemma \ref{Conversion}, this idea extends to $\bar M_\infty$ in the following way.  For any number field $K$,
the set $\rad(K)$ contains the collection of all roots of unity. Therefore, we may write
\begin{equation*}
	S(K) = \rad(K)/\tor(\algt) = \pi(\rad(K))
\end{equation*}
and note that $S(K)$ is a subspace of $V$.  We need only use elements of $S(K_\alpha)$ in order to compute the value of $\bar M_\infty(\alpha)$.

\begin{thm} \label{Reduction}
	Let $\alpha$ be a non-zero algebraic number and assume that $\bar\alpha_1,\ldots,\bar\alpha_N\in V$ satisfy $\pi(\alpha) = \bar\alpha_1\cdots\bar\alpha_N$.
	Then there exists a representation $\pi(\alpha) = \bar\beta_1\cdots\bar\beta_N$ such that
	\begin{equation*}
		\bar\beta_n\in S(K_\alpha)\quad\mathrm{and}\quad \bar M(\bar\beta_n) \leq \bar M(\bar\alpha_n)
	\end{equation*}
	for all $n$.
\end{thm}
\begin{proof}
	We noted in the introduction that the infimum in the definition of $\bar M$ is always attained.  
	Hence, we may choose points $\alpha_n\in \pi^{-1}(\bar\alpha_n)$ such that $\bar M(\bar\alpha_n) = M(\alpha_n)$ for each $n$.
	It follows that $\pi(\alpha) = \pi(\alpha_1\cdots\alpha_N)$ so there exists a root of unity $\xi$ such that
	\begin{equation*}
		\alpha = \xi\alpha_1\cdots\alpha_N.
	\end{equation*}
	By Theorem 2.1 of \cite{Samuels}, there exists another root of unity $\zeta$ and $\beta_1,\ldots,\beta_N$ satisfying the three conditions
	\begin{enumerate}[(i)]
		\item $\alpha = \zeta\beta_1\cdots\beta_N$,
		\item $\beta_n \in \rad(K_\alpha)$ for all $n$,
		\item $M(\beta_n)\leq M(\alpha_n)$ for all $n$.
	\end{enumerate}
	Now set $\bar\beta_n = \pi(\beta_n)$ so that the above conditions imply
	\begin{enumerate}[(i)]
		\item $\pi(\alpha) = \bar\beta_1\cdots\bar\beta_N$,
		\item $\bar\beta_n \in \pi(\rad(K_\alpha)) = S(K_\alpha)$ for all $n$,
		\item $\bar M(\bar\beta_n) \leq M(\beta_n)\leq M(\alpha_n) = \bar M(\bar\alpha_n)$ for all $n$
	\end{enumerate}
	and the theorem follows.
\end{proof}

\section{Proof of Theorem \ref{Computation}} \label{MainProof}

Our proof is based upon the following observation.  Although it will be used as a lemma in the proof of our main result, we give it here as a theorem
since we believe its statement has independent interest.

\begin{thm} \label{BestInf}
	Let $\alpha$ be a non-zero algebraic number.  If $\bar\gamma \in S(K_\alpha)$ with $\bar M(\bar\gamma) \leq M(\alpha)$ then there exists a non-zero 
	rational number $r$ such that
	\begin{equation} \label{BestInfEq}
		\bar M(\bar\gamma^r) = \inf\{\bar M(\bar\gamma^s):s\in\rat^\times\}
	\end{equation}
	and $\bar\gamma^r \in \bar B(\alpha)$.
\end{thm}

The significance of Theorem \ref{BestInf} is that the rational number $r$ 
simultaneously achieves the infimum on the right hand side of \eqref{BestInfEq} and forces $\bar \gamma^r\in \bar B(\alpha)$.
Our proof will require a few lemmas, the first of which is simply Lemma 3.1 of \cite{Samuels}.  Although the proof will be omitted, we include the statement 
here because it will be used very frequently throughout the remainder of the paper.

\begin{lem} \label{HeightInK}
	Let $K$ be a Galois extension of $\rat$.  If $\gamma\in\rad(K)$ then there exists a root of unity $\zeta$ and $L,S\in\nat$
	such that
	\begin{equation*}
		\zeta\gamma^L\in K\quad\mathrm{and}\quad M(\gamma) = M(\zeta\gamma^L)^S.
	\end{equation*}
	In particular, the set
	\begin{equation*}
		\{M(\gamma):\gamma\in\rad(K),\ M(\gamma) \leq C\}
	\end{equation*}
	is finite for every $C \geq 1$.
\end{lem}
\begin{proof}
	See Lemma 3.1 of \cite{Samuels}.
\end{proof}

In view of Theorem \ref{Reduction}, it will be important to consider representations of $\pi(\alpha)$ having elements in $S(K_\alpha)$.  
Our next lemma shows that any such element with sufficiently small Mahler measure must always have an integer power in the finite set $\bar B(\alpha)$.

\begin{lem} \label{MoveToB}
	Let $\alpha$ be a non-zero algebraic number.  If $\bar\gamma\in S(K_\alpha)$ with $\bar M(\bar\gamma)\leq M(\alpha)$ then 
	there exists a positive integer $L$ such that $\bar\gamma^L\in \bar\B(\alpha)$.
\end{lem}
\begin{proof}
	We know that the infimum in the definition of $\bar M$ is attained, so we may choose $\gamma\in\pi^{-1}(\bar\gamma)$ such that $M(\gamma) = \bar M(\bar\gamma)$.  
	This means also that $\gamma\in\rad(K_\alpha)$, so Lemma \ref{HeightInK} gives the existence of a root of unity $\zeta$ and $L,S\in\nat$ such that
	\begin{equation*}
		\zeta\gamma^L \in K^\times\quad\mathrm{and}\quad M(\gamma) = M(\zeta\gamma^L)^S.
	\end{equation*}
	Hence,
	\begin{equation*}
		M(\zeta\gamma^L) \leq M(\zeta\gamma^L)^S = \bar M(\bar\gamma) \leq M(\alpha)
	\end{equation*}
	and it follows that $\zeta\gamma^L\in\B(\alpha)$ from the defintion of $\B(\alpha)$.  Now we may conclude that
	\begin{equation*}
		\bar\gamma^L = \pi(\gamma)^L = \pi(\zeta\gamma^L) \in \bar \B(\alpha)
	\end{equation*}
	completing the proof.
\end{proof}

In the introduction, we noted that the infimum in the definition of $\bar M$ is always attained.  In the proof of our Theorem, it will be useful to know that
the infimum of the set $\{\bar M(\bar\alpha^s):s\in \rat^\times\}$ is also attained.  We establish this fact in the following lemma.

\begin{lem} \label{InfLemma}
	If $\bar\gamma\in V$ then there exists a non-zero rational number $r$ such that
	\begin{equation} \label{InfimumPowerLemmaEq}
		\bar M(\bar\gamma^r) = \inf\{\bar M(\bar\gamma^s):s\in \rat^\times\}.
	\end{equation}
\end{lem}
\begin{proof}
	Let $C$ be a positive real number strictly greater than the right hand side of \eqref{InfimumPowerLemmaEq}.  We claim that
	\begin{equation}\label{InRad}
		\{\bar M(\bar\gamma^s): s\in\rat^\times,\ \bar M(\bar\gamma^s)\leq C\} \subseteq \{M(\delta):\delta\in\rad(K_\gamma),\ M(\delta)\leq C\}.
	\end{equation}
	To see this, assume that $s\in\rat^\times$ is such that $\bar M(\bar\gamma^s) \leq C$ and let $\gamma\in \pi^{-1}(\bar\gamma)$.  
	Now select integers $s_1$ and $s_2$, with $s_2\ne 0$, such that $s = s_1/s_2$
	and choose $\beta\in\alg^\times$ such that $\beta^{s_2} = \gamma^{s_1}$.  It follows immediately that $\beta\in\rad(K_\gamma)$.
	
	We know that $\pi(\beta) = \bar\gamma^s$, which implies that $\bar M(\beta) = \bar M(\bar\gamma^s)$.  Further, there must
	exist a root of unity $\zeta$ such that $\bar M(\beta) = M(\zeta\beta)$.  Therefore, we have that
	\begin{equation*}
		\bar M(\bar\gamma^s) = M(\zeta\beta).
	\end{equation*}
	However, since $\beta\in\rad(K_\gamma)$, we know that $\zeta\beta$ also belongs to $\rad(K_\gamma)$ verifying \eqref{InRad}.

	According to Lemma \ref{HeightInK}, this set is finite implying that the left hand side of \eqref{InRad} is also finite.  Hence, its infimum must
	always be attained, completing the proof.
\end{proof}

Before we proceed with the proof of Theorem \ref{BestInf}, we pause momentarily to examine our results so far.  In view of Lemma \ref{MoveToB}, we have shown that there exists
a {\bf positive integer} $L$ such that $\bar\gamma^L\in \bar B(\alpha)$.  Furthermore, Lemma \ref{InfLemma} shows that the infimum on the right hand side of
\eqref{BestInfEq} is attained.  In other words, we have already established the two conclusions of Theorem \ref{BestInf} for {\bf possibly distinct} rational numbers.
It remains to prove the existence of a non-zero rational number $r$ that {\bf simultaneously} achieves the infimum on the right hand side of \eqref{BestInfEq} and forces 
$\bar \gamma^r\in \bar B(\alpha)$.

\begin{proof}[Proof of Theorem \ref{BestInf}]
	By Lemma \ref{MoveToB}, there exists a postive integer $L'$ such that $\bar\gamma^{L'}\in\bar B(\alpha)$.  Set $\bar b = \bar\gamma^{L'}$ and note that
	by Lemma \ref{InfLemma}, there exists a non-zero rational number $r'$ such that
	\begin{equation} \label{RPrime}
		\bar M(\bar b^{r'}) = \inf\{M(\bar b^s):s\in\rat^\times\}.
	\end{equation}
	Now choose a point $b\in\B (\alpha)$ such that $\pi(b) = \bar b$ and non-zero integers $s'$ and $t'$ such that $r' = s'/t'$.  Further, we select
	a non-zero algebraic number $c$ such that
	\begin{equation} \label{CEquation}
		c^{t'} = b^{s'}
	\end{equation}
	which yields immediately
	\begin{equation} \label{OtherCEquation}
		\pi(c) = \bar b^{r'}.
	\end{equation}
	We know that there exists a root of unity $\zeta$ such that
	\begin{equation*}
		M(\zeta c) = \bar M(c) = \bar M(\pi(c))
	\end{equation*}
	which yields immediately
	\begin{equation*} 
		M(\zeta c) = \bar M(\bar b^{r'}).
	\end{equation*}
	Furthermore, we have that
	\begin{equation*}
		\bar M(\bar b^{r'}) = \inf\{M(\bar b^s):s\in\rat^\times\} \leq \bar M(\bar b) = \bar M(b) \leq M(b).
	\end{equation*}
	Then since $b\in \B(\alpha)$, we may apply the definition of $\B (\alpha)$ to see that $M(b) \leq M(\alpha)$ implying that
	\begin{equation} \label{BoundedByAlpha}
		M(\zeta c) = \bar M(\bar b^{r'}) \leq M(\alpha).
	\end{equation}
	
	Of course, we also have that $b\in K_\alpha$ since $b\in \B(\alpha) \subseteq K_\alpha$.  So by \eqref{CEquation} we conclude that $c\in\rad(K_\alpha)$
	and we obtain immediately that $\zeta c\in \rad(K_\alpha)$.
	By Lemma \ref{HeightInK}, there exists another root of unity $\xi$, as well as $L, S\in \nat$, such that
	\begin{equation} \label{ZetaC}
		\xi(\zeta c)^L\in K_\alpha\quad\mathrm{and}\quad M(\zeta c) = M(\xi(\zeta c)^L)^S.
	\end{equation}
	Hence we apply \eqref{BoundedByAlpha} to conlude that
	\begin{equation*}
		M(\xi(\zeta c)^L) \leq M(\xi(\zeta c)^L)^S = M(\zeta c) \leq M(\alpha)
	\end{equation*}
	which implies immediately, by the defintion of $\B(\alpha)$, that $\xi(\zeta c)^L \in \B(\alpha)$.
	Using \eqref{OtherCEquation} and the fact that $\pi$ is a group homomorphism, we find that
	\begin{equation*} 
		\bar b^{r'L} = \pi(c)^L = \pi(\xi(\zeta c)^L) \in \bar\B(\alpha).
	\end{equation*}
	Now set $r_0 = r'L$ so we have that
	\begin{equation} \label{FinalContainment}
		\bar b^{r_0}\in \bar\B(\alpha).
	\end{equation}
	
	It is obvious that
	\begin{equation*}
			\bar M(\bar b^{r_0}) \geq \inf\{M(\bar b^s):s\in\rat^\times\}
	\end{equation*}
	since $r_0\in \rat^\times$.  Using again \eqref{OtherCEquation}, we see that $\bar b^{r_0} = \pi(c)^L$ and we obtain
	\begin{equation*}
		\bar M(\bar b^{r_0})  = \bar M(\pi(c)^L) = \bar M(\pi(\xi(\zeta c)^L)) \leq M(\xi(\zeta c)^L) \leq M(\xi(\zeta c)^L)^S.
	\end{equation*}
	Then we apply the right hand side of \eqref{ZetaC} as well as \eqref{BoundedByAlpha} to find that
	\begin{equation*}
		M(\xi(\zeta c)^L)^S = M(\zeta c)= \bar M(\bar b^{r'})
	\end{equation*}
	which, by \eqref{RPrime}, yields
	\begin{equation*}
		\bar M(\bar b^{r_0}) \leq \bar M(\bar b^{r'})= \inf\{M(\bar b^s):s\in\rat^\times\}.
	\end{equation*}
	We have finally shown that
	\begin{equation} \label{FinalInfimum}
		\bar M(\bar b^{r_0}) = \inf\{M(\bar b^s):s\in\rat^\times\}.
	\end{equation}
	
	Now replacing $\bar b$ by $\bar\gamma^{L'}$ in both \eqref{FinalContainment} and \eqref{FinalInfimum}, we obtain that
	$\bar\gamma^{r_0 L'} \in \bar\B(\alpha)$ and
	\begin{equation*}
		\bar M(\bar \gamma^{r_0 L'}) = \inf\{M(\bar \gamma^{L's}):s\in\rat^\times\} = \inf\{M(\bar \gamma^{s}):s\in\rat^\times\}.
	\end{equation*}
	Then setting $r = r_0 L'$ we complete the proof.
\end{proof}

We are now prepared to prove Theorem \ref{Computation}.

\begin{proof}[Proof of Theorem \ref{Computation}]
According to Lemma \ref{Conversion}, we have that $M_\infty(\alpha) = \bar M_\infty(\pi(\alpha))$, so it is enough to show that
$\bar M_\infty(\pi(\alpha)) = \bar M(\bar b_J)$.

We proceed by proving that $\bar M_\infty(\pi(\alpha)) \geq \bar M(\bar b_J)$.  To see this, assume that $\bar\alpha_1,\ldots,\bar\alpha_N\in V$
are such that $\pi(\alpha) = \bar\alpha_1\cdots\bar\alpha_N$.  We will show that
\begin{equation} \label{NonInfLower}
	\max\{\bar M(\bar\alpha_1),\ldots,\bar M(\bar\alpha_N)\} \geq \bar M(\bar b_J).
\end{equation}
We may assume that $\bar M(\bar\alpha_n) \leq \bar M(\pi(\alpha))$, and moreover, by Theorem \ref{Reduction} we may assume without loss of generality that 
$\bar\alpha_1,\ldots,\bar\alpha_N\in S(K_\alpha)$.

By Theorem \ref{BestInf}, there exist non-zero rational numbers $r_n$ such that.
\begin{equation} \label{HomeOfAlpha}
	\bar\alpha_n^{r_n}\in\bar\B(\alpha)\quad\mathrm{and}\quad \bar M(\bar\alpha_n^{r_n}) = \inf\{\bar M(\bar\alpha_n^{s}): s\in \rat^\times\}.
\end{equation}
Hence, for each $n$ there exists an index $\ell_n$ such that $\bar\alpha^{r_n} = \bar b_{\ell_n}$ so we have that
\begin{equation*}
	\pi(\alpha) = \prod_{n=1}^N \bar b_{\ell_n}^{1/r_n}.
\end{equation*}
Therefore, if $\ell_n < J$ for all $n$, then $\pi(\alpha) \in \spa\{\bar b_1,\ldots,\bar b_{J-1}\}$, a contradiction.  So there must exist some index $m$
such that $\ell_m \geq J$.  Then by our ordering of elements in $\bar\B(\alpha)$ we get that $\bar M(\bar b_{\ell_m}) \geq \bar M(\bar b_J)$.
It then follows from the right hand side of \eqref{HomeOfAlpha} that
\begin{equation*}
	\bar M(\bar\alpha_m) \geq \bar M(\bar\alpha_m^{r_m}) = \bar M(\bar b_{\ell_m}) \geq \bar M(\bar b_J)
\end{equation*}
verifying \eqref{NonInfLower}.   Now take the infimum of both sides of \eqref{NonInfLower} over all representations of $\pi(\alpha)$ so that we obtain
$\bar M_\infty(\pi(\alpha)) \geq \bar M(\bar b_J)$.

We must now verify that $\bar M_\infty(\pi(\alpha)) \leq \bar M(\bar b_J)$.  We know that $$\pi(\alpha) \in \spa\{\bar b_1,\cdots,\bar b_{J}\}$$ so that there exist
non-zero rational numbers $r_j$ such that
\begin{equation*}
	\pi(\alpha) = \bar b_1^{r_1}\cdots \bar b_J^{r_J}.
\end{equation*}
Now write $r_j = s_j/t_j$ where $a_i\in\nat$ and $t_j\in \intg\setminus\{0\}$ for all $j$.  Now we have that
\begin{equation*}
	\pi(\alpha) = \prod_{j=1}^J \bar b_j^{s_j/t_j} = \prod_{j=1}^J \underbrace{\bar b_j^{1/t_j}\cdots \bar b_j^{1/t_j}}_{s_j \mathrm{times}}
\end{equation*}
implying that
\begin{equation} \label{UpperBound1}
	\bar M_\infty(\pi(\alpha)) \leq \max\{\bar M(\bar b_1^{1/t_1}),\ldots,\bar M(\bar b_J^{1/t_J})\}.
\end{equation}
By Lemma \ref{IntegerPower}, we have that $\bar M(\bar b_j^{1/t_j}) \leq \bar M(\bar b_j)$ for all $j$ so that \eqref{UpperBound1} implies that
\begin{equation*}
	\bar M_\infty(\pi(\alpha)) \leq \max\{\bar M(\bar b_1),\ldots,\bar M(\bar b_J)\} = \bar M(\bar b_J)
\end{equation*}
completing the proof.
\end{proof}

We continue now with the proof of Corollary \ref{SimpleComputation}, our modified version of Theorem \ref{Computation}.

\begin{proof}[Proof of Corollary \ref{SimpleComputation}]
	We note first that $\pi(\B) = \bar \B(\alpha)$ which implies immediately that
	\begin{equation*}
		\bar \B(\alpha) = \{\pi(b_1),\ldots,\pi(b_N)\}
	\end{equation*}
	and we clearly have that
	\begin{equation*}
		\bar M(\pi(b_1)) \leq \cdots \leq \bar M(\pi(b_N)).
	\end{equation*}
	We know there exist integers $r_1,\ldots, r_J$ and a positive integer $s$ such that $\alpha^s = b_1^{r_1}\cdots b_J^{r_J}$.  This means that
	\begin{equation*}
		\pi(\alpha) = \pi(b_1)^{r_1/s}\cdots \pi(b_J)^{r_J/s}
	\end{equation*}
	and $\pi(\alpha) \in \spa\{\pi(b_1),\ldots,\pi(b_J)\}$.
	
	If $\pi(\alpha) \in \spa\{\pi(b_1),\ldots,\pi(b_{J-1})\}$ then there exists integers $r_1,\ldots, r_{J-1}$ and positive integers $s_1,\ldots,s_{J-1}$
	such that
	\begin{equation*}
		\pi(\alpha) = \pi(b_1)^{r_1/s_1}\cdots \pi(b_{J-1})^{r_{J-1}/s_{J-1}}.
	\end{equation*}
	Setting $s' = s_1\cdots s_{J-1}$ and 
	\begin{equation*}
		k_j = r_j\prod_{i\ne j} s_i
	\end{equation*}
	we find that $\pi(\alpha)^{s'} = \pi(b_1)^{k_1}\cdots \pi(b_{J-1})^{k_{J-1}}$.  This yields immediately that there exists a root of unity $\zeta$ such that
	\begin{equation*}
		\zeta\alpha^{s'} = b_1^{k_1}\cdots b_{J-1}^{k_{J-1}}.
	\end{equation*}
	choosing $\ell$ such that $\zeta^{\ell} = 1$ and setting $s = \ell s'$, we see that $\alpha^{s} \in \langle b_1,\ldots,b_{J-1} \rangle$, a contradiction.   Hence, we get that
	\begin{equation*}
		\pi(\alpha) \not\in \spa\{\pi(b_1),\ldots,\pi(b_{J-1})\}.
	\end{equation*}
	By Theorem \ref{Computation}, we conclude that $M_\infty(\alpha) = \bar M(\pi(b_J)) = \bar M(b_J)$.
\end{proof}

Before we continue with the proof of our results in section \ref{Applications}, we establish Corollary \ref{RationalAlg}.

\begin{proof}[Proof of Corollary \ref{RationalAlg}]
	Suppose $\alpha = a/b$ with $a$ relatively prime to $b$.  Hence, we know that $M(\alpha) = \max\{|a|,|b|\}$.  Let
	\begin{equation*}
		\B = \left\{ \frac{c}{d} : \gcd(c,d) = 1\ \mathrm{and}\ \max\{|c|,|d|\} \leq M(\alpha)\right\}
	\end{equation*}
	and note that $\pi(\B) = \bar B(\alpha)$.  Also, it is clear that $\bar M(c/d) = M(c/d)$ for all $c/d\in\B$
	since the degree of $c/d$ is already as small as possible.
	
	Suppose that $p$ is the largest prime dividing $a$ or $b$.  Of course, $p\leq M(\alpha)$ so that $p\in \B$.  Now assume that
	\begin{equation*}
		\B = \{b_1,\ldots,b_J = p,\ldots,b_N\}
	\end{equation*}
	with
	\begin{equation} \label{RationalOrdering}
		M(b_1) \leq \cdots \leq M(b_N).
	\end{equation}
	We may assume further that $M(b_{J-1}) < M(b_J) = p$, because otherwise, we may switch $b_{J-1}$ and $b_J$ while still satisfying \eqref{RationalOrdering}.
	We can repeat this process until $M(b_{J-1}) < M(b_J)$.
	
	Since $p$ is the largest prime dividing $a$ or $b$, it is clear that all primes dividing $a$ or $b$ must appear in $\{b_1,\ldots,b_J \}$.
	Hence, we have that $\alpha\in \langle b_1,\ldots,b_J \rangle$.
	
	Now assume that there exists a non-zero integer $s$
	such that $\alpha^s \in \langle b_1,\ldots,b_{J-1} \rangle$.  Since $M(b_j) < p$ whenever $1\leq j\leq J-1$, we know that $|b_j|_p < p$ by definition of the Mahler measure.
	This implies immediately that $|b_j|_p \leq 1$.  Similarly, we have that $M(b_j^{-1}) < p$ whenever $1\leq j\leq J-1$, so that $|b_j^{-1}|_p < p$.  It now follows that
	$|b_j|_p = 1$.  Since $\alpha^s\in \langle b_1,\ldots,b_{J-1} \rangle$, we conclude that $|\alpha|_p = 1$ a contradiction.  Therefore, we must have that
	$\alpha^s \not\in \langle b_1,\ldots,b_{J-1} \rangle$ for any positive integer $s$.
	
	Finally, it follows from Corollary \ref{SimpleComputation} that $M_\infty(\alpha) = M(b_J) = p$.
\end{proof}

\section{Proofs from section \ref{Applications}} \label{ApplicationsProofs}

We begin our proof of Theorem \ref{RestrictedListComputing} with a lemma.

\begin{lem} \label{BAlg}
	Suppose that $\alpha$ is a non-zero algebraic number.  If $\gamma\in \B_N(\alpha)$ has minimal polynomial $a_Nx^N + \cdots a_1x + a_0$ over $\intg$ then
	\begin{equation*}
		|a_n| \leq \binom{N}{n} M(\alpha).
	\end{equation*}
\end{lem}
\begin{proof}
	If $I = \{n_1,\ldots,n_M\}\subseteq \{1,\ldots,N\}$ then we write 
	\begin{equation*}
		\gamma_I = \prod_{m=1}^M \gamma_{n_m}
	\end{equation*}
	and notice immediately that
	\begin{equation} \label{GammaConsolidation}
		|\gamma_I| = \left| \prod_{m=1}^M \gamma_{n_m} \right| \leq \prod_{m=1}^M \max\{1,|\gamma_{n_m}|\}
			\leq \prod_{n=1}^N \max\{1,|\gamma_{n}|\}.
	\end{equation}
	Using our above notation, we may write the elementary symmetric polynomials evaluated at $\gamma_1,\ldots,\gamma_N$ as
	\begin{equation*}
		e_i(\gamma_1,\ldots,\gamma_N) = \sum_{|I| = i}\gamma_I,
	\end{equation*}
	where $0\leq i\leq N$.  Using \eqref{GammaConsolidation}, we find that
	\begin{align*}
		|e_i(\gamma_1,\ldots,\gamma_N)| & \leq \sum_{|I| = i}|\gamma_I| \\
			& \leq  \sum_{|I| = i}\left(\prod_{n=1}^N \max\{1,|\gamma_{n}|\}\right) \\
			& = \binom{N}{i}\cdot\prod_{n=1}^N \max\{1,|\gamma_{n}|\}\\
			& = \binom{N}{i}\cdot\frac{M(\gamma)}{|a_N|}
	\end{align*}
	It is well-known that $a_n = a_N\cdot e_n(\gamma_1,\ldots,\gamma_N)$ for every $n$ so that
	\begin{equation*}
		|a_n| \leq \binom{N}{n} M(\gamma) \leq \binom{N}{n} M(\alpha).
	\end{equation*}
	where the last inequality follows from the fact that $\gamma\in \B(\alpha)$.
\end{proof}
	
Before we proceed with the proof of Theorem \ref{RestrictedListComputing}, we make one remark regarding Lemma \ref{BAlg}.
In \eqref{BObstacle}, we noted the existence of a highly inefficient method for writing down polynomials whose roots belong to $\B(\alpha)$.
Indeed, Lemma \ref{BAlg} shows that all points $\gamma\in \B(\alpha)$, regardless of their degree, have a minimal polynomial belonging to the set
\begin{equation} \label{BContain}
	\left\{ a_Nx^N +\cdots + a_0 \in \intg[x]: a_N > 0,\ N | [K_\alpha:\rat],\ \mathrm{and}\ |a_i| \leq \binom{N}{i} M(\alpha) \right\}.
\end{equation}
It is theoretically possible to search \eqref{BContain} for polynomials of Mahler measure at most $M(\alpha)$.  However, Lemma \ref{BAlg} still fails to
address the obstacles presented in \eqref{BObstacle}.  After all, the set given by \eqref{BContain} has cardinality
\begin{equation*}
	\sum_{N| [K_\alpha:\rat]} N\prod_{n=2}^N \left(2\binom{N}{n} + 1\right)
\end{equation*}
which is at least
\begin{equation*}
	\exp\left( [K_\alpha:\rat]\log [K_\alpha:\rat]\right).
\end{equation*}
In general, we know only that $\deg\alpha \leq [K_\alpha:\rat] \leq (\deg\alpha)!$ so that an efficient algorithm seems hopeless for $\alpha$ of large degree.  
Furthermore, as we noted in \eqref{BObstacle}, this provides only a list of polynomials, whereas we need their roots.  These polynomials do not necessarily generate
solvable extensions of $\rat$, so it seems far out of reach to attempt to record their roots.

Fortunately, in the case where $\alpha$ has degree $2$, Lemma \ref{BAlg} is efficient enough that we could write our program {\tt B2List(A,B,C,k)}.
Furthermore, the roots of the resulting polynomials are simple to calculate using the quadratic formula.

Our next lemma is a consequence of a result of Dubickas \cite{Dubickas}, finding a positive constant $c = c(\alpha)$ such that
$(M(\alpha),M(\alpha) + c)$ contains no Mahler measures of points in $\rat(\alpha)$.

\begin{lem} \label{DubickasLemma}
	Suppose $\alpha$ is an algebraic number of degree at most $2$.  If $\gamma\in \rat(\alpha)$ satisfies $M(\gamma) > M(\alpha)$ then
	\begin{equation*}
		M(\gamma) > M(\alpha) + \frac{1}{16 M(\alpha)^4}.
	\end{equation*}
\end{lem}
\begin{proof}
	If $\alpha$ is rational then the result is obvious so we will assume that $\deg\alpha = 2$.  Now assume that $\gamma\in \rat(\alpha)$ with $M(\gamma) > M(\alpha)$.
	
	Adopting the notation of \cite{Dubickas}, for an algebraic number $\beta$ with conjugates $\beta_1,\ldots,\beta_N$ over $\rat$, we will write 
	$\overline{|\beta|} = \max\{|\beta_1|,\ldots,|\beta_N|\}$.
	By Theorem 1 of \cite{Dubickas}, we know that
	\begin{equation} \label{DistanceBeginning}
		\log(M(\gamma) - M(\alpha)) > -\deg(M(\gamma))\cdot\log\left(2^{\deg(M(\alpha))}\cdot A\cdot\overline{|M(\alpha)|}^{\deg(M(\alpha))}\right)
	\end{equation}
	where $A$ denotes the leading coefficient of the minimal polynomial of $M(\alpha)$ over $\intg$.  It is well-known that $M(\alpha)$ is an algebraic integer
	so that $A = 1$.  Then simplifying \eqref{DistanceBeginning}, we obtain that
	\begin{equation*}
		M(\gamma) - M(\alpha) > \left( 2\cdot \overline{|M(\alpha)|}\right)^{-\deg M(\alpha)\cdot\deg M(\gamma)}.
	\end{equation*}
	It is easily verified that $M(\alpha)$ and $M(\gamma)$ both belong to $\rat(\alpha)$ so that
	\begin{equation*}
		-\deg M(\alpha)\cdot\deg M(\gamma) \geq -4,
	\end{equation*}
	which yields
	\begin{equation} \label{DistanceEnd}
		M(\gamma) - M(\alpha) > \left( 2\cdot \overline{|M(\alpha)|}\right)^{-4} = \frac{1}{16\cdot \overline{|M(\alpha)|}^4}.
	\end{equation}
	
	We now claim that $\overline{|M(\alpha)|} = M(\alpha)$.  To see this assume that $\alpha_1$ and $\alpha_2$ are the conjugates of $\alpha$ over $\rat$
	and that $\alpha$ has minimal polynomial $ax^2 + bx+c$ over $\intg$, with $a > 0$.  We consider three cases.
	
	If both conjugates of $\alpha$ lie inside the closed unit disk then $M(\alpha) = a\in \intg$.  It now follows that 
	\begin{equation*}
		\overline{|M(\alpha)|} = \overline{|a|} = a = M(\alpha).
	\end{equation*}
	If both conjugates of $\alpha$ lie strictly outide the closed unit disk, then $M(\alpha) = a\cdot|\alpha_1|\cdot|\alpha_2| = |c| \in \intg$.  In this case, we have that
	\begin{equation*}
		\overline{|M(\alpha)|} = \overline{||c||} = |c| = M(\alpha).
	\end{equation*}
	
	Finally, assume without loss of generality that $|\alpha_1| > 1$ and $|\alpha_2| \leq 1$.  Here, we have that $M(\alpha) = a\cdot|\alpha_1|$.
	We know that $a\cdot|\alpha_1|$ is degree $2$ and has conjugate $\pm a\cdot|\alpha_2|$ over $\rat$.  Clearly, we have that
	\begin{equation*}
		|a\cdot|\alpha_1|| > |\pm a\cdot|\alpha_2||
	\end{equation*}
	which implies that
	\begin{equation*}
		\overline{|M(\alpha)|} = \overline{|a\cdot|\alpha_1||} = |a\cdot|\alpha_1|| = a\cdot|\alpha_1| = M(\alpha)
	\end{equation*}
	establishing our claim.  The lemma now follows from \eqref{DistanceEnd}.
\end{proof}

We now proceed with our proof of Theorem \ref{RestrictedListComputing}.

\begin{proof}[Proof of Theorem \ref{RestrictedListComputing}]
	Using \eqref{RationalB}, it is immediately clear that $\B' \subseteq \B_1(\alpha)$ and $\pi(\B') = \pi(\B_1(\alpha))$.  We will now show that
	\begin{equation} \label{B2Comp}
		\B_2(\alpha) = \{\pm\gamma:\gamma\in \B''\}.
	\end{equation}
	
	We begin by taking $\gamma\in \B''$ and showing that $\gamma \in\B_2(\alpha)$.  We know there exists a polynomial $ax^2+bx+c$ that belongs to the output of 
	{\tt B2List(A,B,C,k)} and has $\gamma$ as a root.  Hence, we know that $b^2-4ac$ is not a perfect square and that $\gcd(a,b,c) = 1$, implying that $ax^2+bx+c$ is the 
	minimal polynomial of $\gamma$ over $\intg$.  This means, in particular, that $\deg\gamma = 2$.  Furthermore, $(b^2-4ac)/k$ is a perfect square so that
	\begin{equation*}
		\gamma = \frac{-b \pm \sqrt{k}\cdot\sqrt{(b^2-4ac)/k}}{2a} \in \rat(\sqrt k) = K_\alpha.
	\end{equation*}
	Since $(a,b,c)$ appears in the output of {\tt B2List(A,B,C,k)}, we also conclude that 
	\begin{equation*}
		{\tt Mahler(a,b,c)} < {\tt Mahler(A,B,C)} + 2\cdot 10^{-10}.
	\end{equation*}
	Applying \eqref{EstimateAmount} to {\tt Mahler(a,b,c)}, we obtain that
	\begin{equation*}
		M(\gamma) < {\tt Mahler(A,B,C)} + 3\cdot 10^{-10}.
	\end{equation*}
	Then applying \eqref{EstimateAmount} to {\tt Mahler(A,B,C)}, we find that
	\begin{equation} \label{SoftBound}
		M(\gamma) < M(\alpha) + 4\cdot 10^{-10}.
	\end{equation}
	If $M(\gamma) > M(\alpha)$ then we may apply Lemma \ref{DubickasLemma} to conclude that 
	\begin{equation*}
		M(\gamma) > M(\alpha) + \frac{1}{16 M(\alpha)^4}
	\end{equation*}
	which yields immediately
	\begin{equation*}
		\frac{1}{16 M(\alpha)^4} < 4\cdot 10^{-10}
	\end{equation*}
	when combined with \eqref{SoftBound}.  Some simple manipulations lead to the inequality $100 < M(\alpha)$ which is a contradiction.  Thus, we see that
	$M(\gamma) \leq M(\alpha)$ so that $\gamma\in \B_2(\alpha)$ and it follows immediately that $-\gamma\in\B_2(\alpha)$.
	
	Assume now that $\delta\in \B_2(\alpha)$.  We may select $\gamma\in\{\pm \delta\}$ such that the minimal polynomial $ax^2+bx+c$ of $\gamma$ over $\intg$ has $b \geq 0$.
	We claim that $\gamma\in \B''$.  It is clear that $\gamma\in\B_2(\alpha)$, so that by Lemma \ref{BAlg}, we know that
	\begin{equation*}
		|a| \leq M(\alpha),\ |b| \leq 2M(\alpha),\ \mathrm{and}\ |c| \leq M(\alpha).
	\end{equation*}
	Of course, $a, b$ and $c$ are integers so that
	\begin{equation*}
		|a| \leq \lfloor M(\alpha)\rfloor,\ |b| \leq \lfloor 2M(\alpha) \rfloor\ \mathrm{and}\ |c| \leq \lfloor M(\alpha)\rfloor .
	\end{equation*}
	where $\lfloor M(\alpha)\rfloor $ denotes the floor of $M(\alpha)$.  It follows, once again from \eqref{EstimateAmount}, that
	\begin{equation*}
		M(\alpha) < {\tt Mahler(A,B,C)} + 10^{-10}
	\end{equation*}
	so that
	\begin{equation*}
		|a|, |c| \leq \lfloor {\tt Mahler(A,B,C)} + 10^{-10} \rfloor
	\end{equation*}
	and
	\begin{equation*}
		|b| \leq \lfloor 2\cdot({\tt Mahler(A,B,C)} + 10^{-10}) \rfloor.
	\end{equation*}
	In the notation of our PARI program {\tt B2List(A,B,C,k)}, we have that
	\begin{equation*}
		|a| \leq {\tt floor(M)},\ |b| \leq {\tt floor(2*M)},\ \mathrm{and}\ |c| \leq {\tt floor(M)}.
	\end{equation*}
	We already know that $b \geq 0$, and since $\deg\delta = 2$, we also know that $a \geq 1$.  It now follows that
	\begin{equation*}
		1 \leq a \leq {\tt floor(M)},\ 0 \leq b \leq {\tt floor(2*M)},\ \mathrm{and}\ -{\tt floor(M)} \leq c \leq {\tt floor(M)}.
	\end{equation*}
	We know that $ax^2 +bx + c$ is irreducible so that $\gcd(a,b,c) = 1$ and $b^2-4ac$ is not a perfect square.  Moreover, $\delta\in \rat(\sqrt k)$
	so that $(b^2-4ac)/k$ is a perfect square.  Furthermore, we obtain that
	\begin{align*}
		{\tt Mahler(a,b,c)} & < M(\gamma) + 10^{-10} \\
			& \leq M(\alpha) + 10^{-10} \\
			& < {\tt Mahler(A,B,C)} + 2\cdot 10^{-10}
	\end{align*}
	which implies that ${\tt Mahler(a,b,c)} < {\tt M} + 10^{-10}$ and shows that $(a,b,c)$ must appear in the output of {\tt B2List(A,B,C,k)}.
	This means that $\gamma\in\B''$ so that $\delta\in\{\pm\gamma:\gamma\in \B''\}$ verifying \eqref{B2Comp}.
	
	In view of \eqref{B2Comp}, it is clear that $\B'' \subseteq \B_2(\alpha)$ and $\pi(\B'') = \pi(\B_2(\alpha))$.  This yields immediately
	that $\B \subset \B(\alpha)$ and
	\begin{equation*}
		\pi(B) = \pi(\B'\cup\B'') = \pi(\B')\cup\pi(\B'') = \pi(\B(\alpha))
	\end{equation*}
	completing the proof.
		
\end{proof}

Finally, we must prove Theorem \ref{MBarCalcs}.

\begin{proof}[Proof of Theorem \ref{MBarCalcs}]
	To prove \eqref{GeneralCase}, we assume that $\bar M(\gamma) < M(\gamma)$ so there must exist a root of unity $\zeta$ such that
	$M(\zeta\gamma) < M(\gamma)$.  In view of \eqref{MBarPower}, we conclude that
	\begin{equation} \label{DegreeReduction}
		\deg(\zeta\gamma) < \deg(\gamma) \leq 2
	\end{equation}
	so that $\zeta\gamma\in\rat$.  It follows imediately that $\rat(\zeta) = \rat(\gamma)$ and that $\deg\zeta = 2$.
	It is well-known that there are only $6$ degree $2$ roots of unity, and they are
	\begin{equation*}
		\pm i\ \mathrm{and}\ \frac{\pm 1\pm \sqrt{-3}}{2},
	\end{equation*}
	which implies that $\rat(\zeta) = \rat(i)$ or $\rat(\zeta) = \rat(i\sqrt 3)$.  We now have that $\rat(\gamma) = \rat(i)$ or $\rat(\gamma) = \rat(i \sqrt 3)$, 
	a contradiction.
	
	To prove \eqref{ICase}, first assume that $a= 0$ so that $\gamma = bi$ and $\bar M(\gamma) = \bar M(b)$.  But by \eqref{GeneralCase}, we know that
	$\bar M(b) = M(b)$ verifying that $\bar M(\gamma) = M(b)$.
	
	Now assume that $a \ne 0$ and $\bar M(\gamma) < M(\gamma)$. Once again, there exists a rational number $r$ and a root of unity $\zeta$ such that
	\begin{equation*}
		\zeta\gamma = r
	\end{equation*}
	implying that $\zeta$ must be irrational.  Since $\gamma\in\rat(i)$, we conclude that $\zeta\in\rat(i)$ which establishes that $\zeta = \pm i$.
	This yields $\gamma = \pm ri$ which implies that $a= 0$, another contradiction.
	
	To establish \eqref{3rdCase}, suppose first that $a = b$.  Then we have that
	\begin{equation*}
		\gamma = a + a\sqrt{-3} = 2a\left(\frac{1+\sqrt{-3}}{2}\right)
	\end{equation*}
	But, $(1+\sqrt{-3})/2$ is a root of unity.  Taking $\zeta = 2/(1+\sqrt{-3})$ gives $\deg(\zeta\gamma) = 1$, so that by \eqref{MBarPower}, 
	we obtain that $\bar M(\gamma) = M(\zeta\gamma) = M(2a)$.  Similarly, if $a = -b$ then we obtain
	\begin{equation*}
		\gamma = a - a\sqrt{-3} = 2a\left(\frac{1-\sqrt{-3}}{2}\right)
	\end{equation*}
	so that $\bar M(\gamma) = M(2a)$.
	
	Now we assume that $a\not\in \{b,-b\}$ and set $r = a-b\in\rat\setminus\{0,-2b\}$.  This yields immediately that
	\begin{equation} \label{GammaRep}
		\gamma = r + 2b\left(\frac{1+\sqrt{-3}}{2}\right)
	\end{equation}
	so that $r$ and $2b$ are the unique rational numbers $x$ and $y$ such that $\gamma = x + y(1+\sqrt{-3})/2$.  Once again, we assume that
	$\bar M(\gamma) < M(\gamma)$ so there exists a root of unity $\zeta$ such that $\zeta\gamma = s$ is rational.  Hence we have that
	$\zeta = s\gamma^{-1} \in \rat(\sqrt{-3})$.  
	
	We know that the primitive third and primitive sixth roots of unity are the only irrational roots of unity in $\rat(\sqrt{-3})$.
	As before, $\zeta$ must be irrational so it, as well as $\zeta^{-1}$, lies among these four roots of unity.  That is, we have that
	\begin{equation*}
		\zeta^{-1}\in \left\{\frac{1+\sqrt{-3}}{2}, \frac{1-\sqrt{-3}}{2}, \frac{-1+\sqrt{-3}}{2}, \frac{-1-\sqrt{-3}}{2}\right\}.
	\end{equation*}
	We know that $\gamma = s\zeta^{-1}$ so this yields four cases.
	\begin{enumerate}[(i)]
	\item $\gamma = s\left(\frac{1+\sqrt{-3}}{2}\right)$ so that, by \eqref{GammaRep}, we get $r=0$ and $s=2b$, a contradiction.
	\item $\gamma = s\left(\frac{-1-\sqrt{-3}}{2}\right)$, giving $r=0$ and $s = -2b$, also a contradiciton.
	\item $\gamma = s\left(\frac{-1+\sqrt{-3}}{2}\right) = -s + s\left(\frac{1+\sqrt{-3}}{2}\right)$ giving $r=-s$ and $s = 2b$.  Consequently, we have
				$r=-2b$, again a contradiction.
	\item $\gamma = s\left(\frac{1-\sqrt{-3}}{2}\right) = s - s\left(\frac{1+\sqrt{-3}}{2}\right)$ yielding, once again, $r=-2b$.
	\end{enumerate}
\end{proof}

\section{Acknowledgments}

I wish to thank Professor D.B. Zagier for suggesting that the problem solved in this paper likely follows from \cite{Samuels}.  I also thank
M. Widmer for observing that \cite{LoherMasser} can be used to provide a upper bound on the cardinality of $\B(\alpha)$.  Finally, I thank the
Max-Planck-Institut f\"ur Mathematik in Bonn, Germany where most of this research took place.

\end{document}